\date{\today}
\documentclass{amsart}
\usepackage{amsmath}
\usepackage{amsfonts}
\usepackage{amscd}
\usepackage{amssymb}
\usepackage{amsthm}
\usepackage{amsopn}
\usepackage{tikz}
\usetikzlibrary{calc}
\usetikzlibrary{patterns}
\usetikzlibrary{matrix}
\usepgflibrary{arrows}
\usetikzlibrary{arrows}
\usetikzlibrary{decorations.pathreplacing}
\usetikzlibrary{decorations.pathmorphing}
\usepackage{epsfig}
\usepackage{enumerate}

\graphicspath{{bettitablegraphics/}}

\theoremstyle{definition}
\newtheorem{theorem}{Theorem}[section]
\newtheorem{definition}[theorem]{Definition}
\newtheorem{lemma}[theorem]{Lemma}

\newtheorem{example}[theorem]{Example}
\newtheorem{proposition}[theorem]{Proposition}

\newtheorem{conjecture}[theorem]{Conjecture}

\newcommand{\reg}{\text{reg}}
\newcommand{\Tor}{\text{Tor}}
\newcommand{\Stab}{\text{Stab}}

\newcommand{\rees}{\mathcal R}

\newcommand{\kk}{\Bbbk}

\newcommand{\betti}[2]{\beta_{#1,#2}}

\newcommand{\supp}{{\text{{\bf supp}}}}

\newcommand{\ith}{i^{\text{th}}}

\title{Stabilization of Betti Tables}

\author{G. Whieldon}
\address{Gwyneth Whieldon \\
Department of Mathematics\\
Cornell University\\
Ithaca, NY  14850
}
\email{whieldon@math.cornell.edu}

\begin{document}


\maketitle
\begin{abstract} Let $I\subseteq R=\kk[x_1,...,x_n]$ be a homogeneous equigenerated ideal of degree $r$.  We show here that the shapes of the Betti tables of the ideals $I^d$ stabilize, in the sense that there exists some $D$ such that for all $d\geq D$, $\betti{i}{j+rd}(I^d)\neq 0\Leftrightarrow \betti{i}{j+rD}(I^D)\neq 0$.  We also produce upper bounds for the stabilization index $\Stab(I)$.  This strengthens the result of Cutkosky, Herzog, and Trung that the Castelnuovo-Mumford regularity $\reg(I^d)$ is eventually a linear function in $d$.
\end{abstract}

\section{Background and Results}
\subsection{Asymptotics of Regularity of $I^d$}
For an ideal $I\subseteq R=\kk[x_1,...,x_n]$, much work has been done on showing that the Castelnuovo-Mumford regularity of $I^d$ is a linear function in terms of $d$ for high powers.  The following theorem is a result of Cutkosky, Herzog and Trung:

\begin{theorem}[\bf Theorem 1.1 in \cite{MR1711319}]  Let $I$ be an arbitrary homogeneous ideal.  Let $r(I)$ denote the maximum degree of the homogeneous generators of $I$.  The following hold:
\begin{enumerate}[(i)]
\item There is a number $e$ such that $\reg(I^d)\leq d\cdot r(I)+e$ for all $d\geq 1.$
\item $\reg(I^d)$ is a linear function for all $d$ large enough.
\end{enumerate}
\end{theorem}

They provide criterion for estimating this $e$ in the case of an equigenerated ideal $I$, i.e. an ideal generated by homogeneous generators of the same degree.  This result generalizes an earlier bound by Swanson giving the existence of $k$ such that
$$\reg(I^d)\leq kd$$
 for ideals generated by monomials in ~\cite{MR1428875}.
 
 Let $I\subseteq R=\kk[x_1,...,x_n]$ be an ideal.  The \emph{graded Betti numbers} of a homogeneous ideal $I$ are given by $\beta_{i,j}(I) = \dim_{\kk} \Tor_i(\kk,I)_j$.  The graded Betti numbers also correspond to the ranks of the free modules in a minimal free resolution of $I$.  We organize this data into the \emph{Betti table of $I$} (in the style of Macaulay 2) displaying $\betti{i}{i+j}(R/I)$ in the $i^{\text{th}}$ column and $j^{\text{th}}$ row, as seen in Example~\ref{ex:stabs}.

We recall the definition of a singly graded equigenerated ideal.
\begin{definition}  We say that an ideal $I=(f_0,f_1,...,f_k)\subset R=\kk[x_1,...,x_N]$ is \emph{equigenerated in degree $r$} if $\deg(f_i)=r$ for all $f_i$.
\end{definition}
Using techniques similar to those in \cite{MR1711319}, \cite{MR1895463}, and \cite{MR2604920}, we produce here a sharper result on the asymptotics of Betti tables of powers $I^d$.
\begin{theorem}[Theorem~\ref{thm:bettistabilization}]  Let $I=(f_0,f_1,...,f_k)\subseteq\kk[x_1,...,x_n]=R$ be an equigenerated ideal of degree $r$.  Then there exists a $D$ such that for all $d>D$, we have
\begin{equation*}
\betti{i}{j+rd}(I^d)\neq 0\Longleftrightarrow \betti{i}{j+rD}(I^D)\neq 0.
\end{equation*}
\end{theorem}
This gives us that the shape of the Betti tables of powers of an ideal $I$ is eventually fixed, translated down by the degree $r$ of the ideal.
\begin{example}\label{ex:stabs}  Let $I={({x}_{3} {x}_{4} {x}_{5},{x}_{1} {x}_{6} {x}_{7},{x}_{3} {x}_{6} {x}_{8},{x}_{1} {x}_{5} {x}_{9},{x}_{2} {x}_{8} {x}_{9})}\subseteq \kk[x_1,...,x_9].$  We consider the Betti diagrams of the resolutions of the first few powers $I^d$ of $I$.  The diagrams have been shifted to only show nonzero Betti numbers in the resolution of $I^d$.
\begin{center}
\begin{tikzpicture}[scale=.2, betti/.style={
execute at begin cell=\node\bgroup,
execute at end cell=\egroup;,
execute at empty cell=\node{$\cdot$};
}] 
\node at (-12,8) {$I$};
\matrix [betti, ampersand replacement=\&] {
-\&1\&2\&3\&4\&5\\
total: \&5\&10\&9\&3\&\\
2: \&5\&\&\&\&\\
3:\&\&6\&\&\&\\
4:\&\&4\&9\&3\&\\
};
\end{tikzpicture}\;\;\;\;\;\;\;\;\;\;\begin{tikzpicture}[scale=.2, betti/.style={
execute at begin cell=\node\bgroup,
execute at end cell=\egroup;,
execute at empty cell=\node{$\cdot$};
}] 
\node at (-12,8) {$I^2$};
\matrix [betti, ampersand replacement=\&] {
-\&1\&2\&3\&4\&5\\
total: \&15\&41\&39\&12\&\\
5: \&15\&\&\&\&\\
6:\&\&33\&12\&\&\\
7:\&\&8\&27\&12\&\\
};
\end{tikzpicture}\\
\vspace{10pt}
\begin{tikzpicture}[scale=.2, betti/.style={
execute at begin cell=\node\bgroup,
execute at end cell=\egroup;,
execute at empty cell=\node{$\cdot$};
}] 
\node at (-12,8) {$I^3$};
\matrix [betti, ampersand replacement=\&] {
-\&1\&2\&3\&4\&5\\
total: \&35\&117\&121\&39\&1\\
8: \&35\&\&\&\&\\
9:\&\&105\&67\&9\&\\
10:\&\&12\&54\&30\&1\\
};
\end{tikzpicture}\;\;\;\;\;\;\;\;\;\begin{tikzpicture}[scale=.2, betti/.style={
execute at begin cell=\node\bgroup,
execute at end cell=\egroup;,
execute at empty cell=\node{$\cdot$};
}] 
\node at (-12,8) {$I^4$};
\matrix [betti, ampersand replacement=\&] {
-\&1\&2\&3\&4\&5\\
total: \&70\&271\&302\&105\&5\\
11: \&70\&\&\&\&\\
12:\&\&255\&212\&45\&\\
13:\&\&16\&90\&60\&5\\
};
\end{tikzpicture}\\
\vspace{10pt}

\begin{tikzpicture}[scale=.2, betti/.style={
execute at begin cell=\node\bgroup,
execute at end cell=\egroup;,
execute at empty cell=\node{$\cdot$};
}] 
\node at (-12,8) {$I^5$};
\matrix [betti, ampersand replacement=\&] {
-\&1\&2\&3\&4\&5\\
total: \&126\&545\&645\&240\&15\\
14: \&126\&\&\&\&\\
15:\&\&525\&510\&135\&\\
16:\&\&20\&135\&105\&15\\
};
\end{tikzpicture}\;\;\;\;\;\begin{tikzpicture}[scale=.2, betti/.style={
execute at begin cell=\node\bgroup,
execute at end cell=\egroup;,
execute at empty cell=\node{$\cdot$};
}] 
\node at (-12,8) {$I^6$};
\matrix [betti, ampersand replacement=\&] {
-\&1\&2\&3\&4\&5\\
total: \&210\&990\&1229\&483\&35\\
17: \&210\&\&\&\&\\
18:\&\&996\&1040\&315\&\\
19:\&\&24\&189\&168\&35\\
};
\end{tikzpicture}
\end{center}
We can see the stabilized shape of the powers of $I^d$ will be:

\begin{center}
\begin{tikzpicture}[scale=.2, betti/.style={
execute at begin cell=\node\bgroup,
execute at end cell=\egroup;,
execute at empty cell=\node{$\cdot$};
}] 
\node at (-12,8) {$I^d$};
\draw[-,black] (-2.5,1.4)--(-0.2,1.4)--(-.2,-1.2)--(6,-1.2)--(6,-3.6)--(8,-3.6)--(8,-6.1)--(-0.2,-6.1)--(-0.2,-1.2)--(-2.5,-1.2)--cycle;
\matrix [betti, ampersand replacement=\&] {
-\&1\&2\&3\&4\&5\\
total: \&$\ast$\&$\ast$\&$\ast$\&$\ast$\&$\ast$\\
3d-1: \&$\ast$\&\&\&\&\\
3d:\&\&$\ast$\&$\ast$\&$\ast$\&\\
3d+1:\&\&$\ast$\&$\ast$\&$\ast$\&$\ast$\\
};
\end{tikzpicture}
\end{center}
\end{example}
Unfortunately, Theorem~\ref{thm:bettistabilization} does not guarantee that powers of our ideals $I^d$ will  have linear resolutions if the resolution of $I^l$ is linear for some $l$ with $d>l$.  As a counterexample, we have the following example (due to Sturmfels):
\begin{example}[\cite{Sturmfels:2000fk}]\label{ex:lineartonot}  Set $I=(def,cef,cdf,cde,bef,bcd,acf,ade)\subseteq \kk[a,b,c,d,e,f]$.  The ideal $I$ has linear resolution and linear quotients with respect to the ordering given above, but $I^2$ fails to be linear.  We include the Betti tables of $I$ and $I^2$ here.
\begin{center}
\begin{tikzpicture}[scale=.2, betti/.style={
execute at begin cell=\node\bgroup,
execute at end cell=\egroup;,
execute at empty cell=\node{$\cdot$};
}] 
\node at (-10,6) {$I$};
\matrix [betti, ampersand replacement=\&] {
-\&0\&1\&2\&3\\
total: \&1\&8\&11\&4\\\
0: \&1\&\&\&\\
1:\&\&\&\&\\
2:\&\& 8 \& 11 \& 4 \\
};
\node at (0,-15){};
\end{tikzpicture}\;\;\;\;\;\;\;\;\;\;\;\;\;\;\;\;\;\;\begin{tikzpicture}[scale=.2, betti/.style={
execute at begin cell=\node\bgroup,
execute at end cell=\egroup;,
execute at empty cell=\node{$\cdot$};
}] 
\node at (-14,11) {$I^2$};
\matrix [betti, ampersand replacement=\&] {
-\&0\&1\&2\&3\&4\&5\&6\\
total: \&1\&36\&85\&79\&38\&10\&1\\\
0: \&1\&\&\&\&\&\&\\
1:\&\&\&\&\&\&\&\\
2:\&\&\&\&\&\&\&\\
3:\&\&\&\&\&\&\&\\
4:\&\&\&\&\&\&\&\\
5:\&\& 36 \& 84 \& 75 \&32\&6\& \\
6:\&\&\&1\&4\&6\&4\&1\\
};
\end{tikzpicture}
\end{center}
\end{example}
More generally, Conca provided a class of ideals $I_k$ which have linear quotients (and hence, linear resolutions) until the $k^{\text{th}}$ power, then have nonlinear resolutions for all powers higher than $k$ \cite{MR2184787}.  This implies that for an ideal $I$, the shapes of Betti tables of $I, I^2, \ldots, I^d$ and $I^{d+1}$ need not satisfy any chain of inclusions, though they eventually stabilize for some $I^D$.

We also provide an upper bound for the Betti numbers of powers of an equigenerated ideal $I$ in terms of the Betti numbers of the Rees ideal of $I$ as follows.
\begin{theorem}[Theorem~\ref{thm:betticalculation}]  Let $I=(f_0,f_1,...,f_k)\subseteq R=\kk[x_1,...,x_N]$ with $f_i$ homogeneous of degree r.  Let $\rees(I)$ be the Rees algebra of $I$ in ring $S=\kk[x_1,...,x_N,w_0,...,w_k]$ with bigrading $\deg(x_i)=(1,0)$ and $\deg(w_i)=(0,1)$.  Then
$$\betti{i}{j+rd}(I^d)\leq\sum_{m=0}^d\binom{d+k-m}{d-m}\betti{i}{(j,m)}(\rees(I))$$
holds for all $i,j,d$.
\end{theorem}
The proof follows from a careful examination of the restriction of a minimal resolution of $\rees(I)$ to bidegrees $(\ast, d)$.  We give the smallest $D$ for which this stabilization occurs a name:
\begin{definition}[Definition~\ref{defn:stab}]  Let $I$ be a homogeneous equigenerated ideal in polynomial ring $R$.  Let the \emph{stabilization index $\Stab(I)$ of $I$} be the smallest such $D$ such that for all $d\geq D$,
\begin{equation*}
\betti{i}{j+rd}(I^d)\neq 0\Longleftrightarrow \betti{i}{j+rD}(I^D)\neq 0.
\end{equation*}
\end{definition}

Finding $\Stab(I)$ in directly in terms algebraic properties of $I$ remains open, although a conjecture for edge ideals will be given in Section~\ref{sec:stabi}.  Areas of future research include producing explicit $\Stab(I)$ for other classes of ideals or providing sharper bounds for $\Stab(I)$ than those included here.

\section{Rees Algebras of Equigenerated Ideals}\label{sec:equigen}

\subsection{Rees Algebras and Degree Restrictions}\label{subsec:reesalgebras}
One common technique used in investigating powers $I^n$of an ideal $I$ involves passing to the Rees algebra of $I$.  The Rees algebra $\rees(I)$ of an ideal $I$ is an object which captures the ideal $I$ and all of its powers.
\begin{definition}
Let $I=(f_0,f_1,...,f_k)\subseteq R=\kk[x_1,...,x_N]$.  The \emph{Rees algebra $\rees(I)$ of $I$} is
$$\rees(I)=R\oplus It\oplus I^2t^2\oplus I^3t^3\oplus\cdots\oplus I^nt^n\oplus\cdots$$
This is occasionally denoted $R[It]$.
\end{definition}
In general, we will use a presentation of $\rees(I)$ as a quotient module of the ring $S=R[w_0,w_1,...,w_k]=\kk[x_1,...,x_N,w_0,w_1,...,w_k]$.
\begin{proposition}[~\cite{MR1275840}]\label{prop:reespresentationgeneral}  Let $I=(f_1,...,f_k)\subseteq R=\kk[x_1,...,x_N]$ and let $\rees(I)$ be its Rees algebra.  Then $\rees(I)=R[w_1,...,w_k]/L=\kk[x_1,...,x_N,w_0,w_1,...,w_k]/L$, with presentation ideal
\begin{equation*}
L=(f_i-w_i t : 1\leq i\leq k)S[t]\cap S.
\end{equation*}
If $S=\kk[x_1,...,x_N,w_1,...,w_k]$, and $\rees(I)=S/L$, then $L$ is the \emph{Rees ideal of $I$}.
\end{proposition}
\subsection{Resolutions and Bigradings of Rees Algebras}
Taking a resolution (with an appropriately chosen bigrading) of $L$ gives resolutions of all powers of $L$, and can be used to bound or explicitly compute Betti numbers $\betti{i}{j}(I^n)$ for all $n$. 

We will assume throughout this paper that $I=(f_0,f_1,...,f_k)$ is an equigenerated ideal of degree $r$ in $R=\kk[x_1,...,x_N]$.  Notationally, we set $\rees(I)=S/L$ with $L$ the Rees ideal of $I$ and $S=\kk[x_1,...,x_N,w_0,w_1,...,w_k]$.\\
\\
We bigrade $\rees(I)$ by $\deg(x_i)=(1,0)$ and $\deg(w_i)=(0,1)$ and take the minimal graded free resolution of $\rees(I)$ with respect to this grading.
$${\mathcal F}:\;\;\;\rees(I)\leftarrow S\leftarrow \bigoplus_{(j,m)}S(-j,-m)^{\betti{1}{(j,m)}}\leftarrow\cdots\leftarrow \bigoplus_{(j,m)}S(-j,-m)^{\betti{p}{(j,m)}}\leftarrow 0.$$
Restricting to the strand $(\ast,d)$, we obtain a (possibly nonminimal) resolution of $I^d$:
\begin{equation}\label{eq:reesstrand}{\mathcal F_d}:\;\;\;I^d\leftarrow S_{(\ast,d)}\leftarrow \bigoplus_{(j,m)}S(-j,-m)^{\betti{1}{(j,m)}}_{(\ast,d)}\leftarrow\cdots\leftarrow \bigoplus_{(j,m)}S(-j,m)_{(\ast,d)}^{\betti{p}{(j,m)}}\leftarrow 0.
\end{equation}
Tensoring this resolution with $\kk$ and taking the homology of the maps computes us $\dim \Tor^R_{i}(\kk,I^d)_{j+rd}=\betti{i}{j+rd}(I^d)$.  This shift in the indices of $\betti{i}{j+rd}(I^d)$ accounts for the shift in grading to agree with that of $R$ while viewing $I^d$ as an $R$ module.\\
\\
Alternately, we could have first tensored with $S/M$ for $M=(x_1,...,x_N)$, taken homology of our maps, then restricted in degrees.  This will give us modules $\Tor^S_i(S/M,\rees(I))_j$, and as these two actions commute, we have that
\begin{align*}
\Tor^S_i(S/M,\rees(I))_{(j,d)}&=\Tor^S_{i}(S/M,I^d)_j\\
&=\Tor^R_{i}(\kk,I^d)_{j+rd}.\\
\end{align*}
Hence we have that all Betti numbers of higher powers can be written in terms of the dimensions of the bigraded modules $\Tor_i^S(S/M,\rees(I))$, given by
$$\betti{i}{j+rd}(I^d)=\dim\Tor_i^S(S/M,\rees(I))_{(j,d)}.$$

\section{Bounds on Betti Numbers of Powers of Ideals}\label{sec:bettibounds}
We resolve the Rees algebra $\rees(I)$ and restrict to fixed $w$-degree strands to produce explicit bounds on the Betti numbers of $I^d$.
\begin{theorem}\label{thm:betticalculation}  Let $I=(f_0,f_1,...,f_k)\subseteq R=\kk[x_1,...,x_N]$ with all $f_i$ homogeneous of degree r.  Let $\rees(I)$ be the Rees algebra of $I$ in ring $S=\kk[x_1,...,x_N,w_0,...,w_k]$ with bigrading $\deg(x_i)=(1,0)$ and $\deg(w_i)=(0,1)$.  Then
$$\betti{i}{j+dr}(I^d)\leq\sum_{m=0}^d\binom{d+k-m}{d-m}\betti{i}{(j,m)}(\rees(I))$$
holds for all $i,j,d$.
\end{theorem}
\begin{proof}  We take a minimal free resolution of $\rees(I)$ and consider the degree restricted strand used in Section~\ref{sec:equigen}:
\begin{equation*}{\mathcal F_d}:\;\;\;I^d\leftarrow S_{(\ast,d)}\leftarrow \bigoplus_{(j,m)}S(-j,-m)_{(\ast,d)}^{\betti{1}{(j,m)}}\leftarrow\cdots\leftarrow \bigoplus_{(j,m)}S(-j,-m)_{(\ast,d)}^{\betti{p}{(j,m)}}\leftarrow 0.
\end{equation*}
Let $T=\kk[w_0,w_1,...,w_k]$ be the polynomial ring in the $w_i$-variables.  Then we can rewrite our bigraded pieces $S(-j,-m)=R(-j)\otimes T(-m)$.  Then in a fixed strand $(\ast,d)$, we have:
\begin{equation*}{\mathcal F_d}:\;\;\;I^d\leftarrow R\otimes T_d\leftarrow \bigoplus_{(j,m)}R(-j)\otimes T(-m)_{d}^{\betti{1}{(j,d)}}\leftarrow\cdots\leftarrow \bigoplus_{(j,m)}R(-j)\otimes T(-m)_d^{\betti{p}{(j,d)}}\leftarrow 0.
\end{equation*}
It remains to count the dimension over $R$ of the $\ith$ module
$$F_i=\bigoplus_{(j,m)}R(-j)\otimes T(m)_d^{\betti{i}{(j,m)}(\rees(I))}$$
in a fixed degree $j+rd$ of the resolution.  Finally, the dimension of $T(-m)_d$ is the number of degree $d-m$ monomials in a polynomial ring in $k+1$ variables, or
$$\binom{d+k-m}{k}.$$
So we have that
$$\betti{i}{j+rd}(I^d)\leq\sum_{m=0}^{d}\binom{d+k-m}{k}\betti{i}{(j,m)}(\rees(I)),$$
proving the theorem.
\end{proof}
This immediately shows that the Betti diagram of $I^d$ sits inside an (appropriately degree shifted) table coming from the Betti diagram of the resolution of $\rees(I)$.  This implies that the number of nonzero graded Betti numbers of $I^d$ is bounded independent of the power $d$.  We refine this rough bound in the following section.

\section{Betti Diagrams of Powers of Stanley-Reisner Ideals $I_\Delta$}\label{sec:bettidiagramsofpowers}
We are now ready to prove the main theorem:
\begin{theorem}[\bf Betti Tables of Powers of Equigenerated Ideals]\label{thm:bettistabilization}\textcolor{white}{aa}\\
Let $I=(f_0,f_1,...,f_k)\subseteq\kk[x_1,...,x_N]=R$ be an equigenerated ideal of degree $r$.  Then there exists a $D$ such that for all $d>D$, we have
\begin{equation*}
\betti{i}{j+rd}(I^d)\neq 0\Longleftrightarrow \betti{i}{j+rD}(I^D)\neq 0.
\end{equation*}
\end{theorem}
\begin{proof}[Proof of Theorem~\ref{thm:bettistabilization}]  From the calculation in Section~\ref{sec:equigen}, we have that
$$\betti{i}{j+rd}(I^d)=\dim\Tor_i^S(S/M,\rees(I))_{(j,d)}.$$
The $\Tor_i(S/M,\rees(I))$ are finitely generated bigraded $S$-modules.  We decompose them into bigraded components in the following way.

Let $M_i:=\Tor_i(S/M,\rees(I))$ and $M_{ij}:=(M_i)_{(j,\ast)}$.  The $M_{ij}$ are finitely generated graded $T$-modules, where $T=\kk[w_0,w_1,...,w_k]$ is the polynomial ring in the $w_i$-variables.  So each $M_{ij}$ has a Hilbert polynomial such that
$$P_{ij}(d):=P_{M_{ij}}(d)=\dim (M_i)_{(j,d)}$$
for all $d\geq d_{ij}$, with $d_{ij}$ the regularity of $M_{ij}$ as a $T$-module.  Hence, for all $d\geq d_{ij}$ and $P_{M_{ij}}$ not identically zero, we have $\betti{i}{j+dr}(I^d)=\dim (M_i)_{(j,d)}=P_{M_{i,j}}(d)>0$.

Note that $D=\max_{i,j}\{D_{ij}\}$ will be an upper bound for $\Stab(I)$, providing such a maximum exists.
\begin{lemma}\label{lem:Mir}  There are only finitely many nonzero $M_{ij}$.
\end{lemma}
\begin{proof}[Proof of Lemma~\ref{lem:Mir}]  That only finitely many $M_j$ are nonzero follows from
$$\betti{i}{j+rd}(I^d)=\dim\Tor_i^S(S/M,{\mathcal R}(I))_{(j,d)}.$$
As the projective dimension of all powers $I^d$ is bounded by $N$ the number of variables in our original ring, $\Tor_i^S(S/M,\rees(I))=0$ for all $i>N$.\\
\\
We now consider a fixed $M_i$.  Theorem~\ref{thm:betticalculation} gave a bound on the Betti numbers of $I^d$ depending on the Betti numbers of $\rees(I)$,
$$\betti{i}{j+rd}(I^d)\leq\sum_{m=0}^{d}\binom{d+k-m}{d-m}\betti{i}{(j,m)}(\rees(I)).$$
As for a fixed $i$, the number of nonzero Betti numbers of $\rees(I)$ must be finite, there can be only finitely many $j$ such that $\betti{i}{(j,m)}(\rees(I))\neq 0$.  This implies that for $j$ outside of this set, $\betti{i}{j+rd}(I^d)\leq \sum_{m=0}^d 0$ for all $d$, which implies $\betti{i}{j+rd}(I^d)=0$.  So $M_{ij}=0$ except for a finite number of cases.\\
\\
This completes the proof of the lemma.
\end{proof}

By Lemma~\ref{lem:Mir}, we have that the maximum
$$D={\displaystyle \max_{i,j}\biggl\{D_{ij}\biggr\}}$$
exists.  Hence, we have that
$$\dim\Tor_i(S/M,\rees(I)))_{(\ast,d)}=P_{M_i}(d)$$
is a polynomial function for all $d>D$.  We note that for all $d>D$,
$$\betti{i}{j+dr}(I^d)=\dim (M_i)_{(j,d)}=P_{M_{i,j}}(d)>0$$
if and only if
$$\betti{i}{j+Dr}(I^D)=\dim (M_i)_{(j,D)}=P_{M_{i,j}}(D)>0,$$
completing the proof.
\end{proof}
The techniques used throughout the proof of Theorem~\ref{thm:bettistabilization} were similar to those seen in \cite{MR2604920}, \cite{MR1711319}, and \cite{MR1428875}, but extend their results to a classification of all possible nonzero graded Betti numbers of powers of an equigenerated ideal $I$.  

\section{Stabilization Index of $I$}\label{sec:stabi}
The bound $D$ produced in Theorem~\ref{thm:bettistabilization} is not sharp, and finding the smallest such $D$, which we will call the \emph{stabilization index of $I$} $\Stab(I)$, in terms of combinatorial data of $I$ is a subject of future research.
\begin{definition}\label{defn:stab}  Let $I$ be a homogeneous ideal equigenerated in degree $r$ in polynomial ring $R$.  Let $\Stab(I)$ be the smallest such $D$ such that for all $d\geq D$,
\begin{equation*}
\betti{i}{j+rd}(I^d)\neq 0\Longleftrightarrow \betti{i}{j+rD}(I^D)\neq 0.
\end{equation*}
\end{definition}
While this is unknown in general, we conjecture here a formula for $\Stab(I_G)$ for edge ideal $I_G$.
\begin{conjecture}Let $I_G=(m_0,m_1,...,m_k)\subseteq\kk[x_1,...,x_N]$ be a square-free monomial ideal with $\cup_i\supp(m_i)=\{x_1,...,x_N\}$.  Then
$$\Stab(I_G)=\min\{n:\text{there exists an ${\bf m}\in I_G^n$ such that $x_i^2|{\bf m}$ for all i.}\}$$
\end{conjecture}
This seems to be related to the Stanley-Reisner complexes of polarization of the powers of the edge ideal, but a clear proof that the Betti diagrams stabilize from the existence of such a generator is still unknown.  Finding a formula for $\Stab(I)$ of other monomial ideals $I$ remains open.
\subsection{Areas of Future Research}
We would like to answer the following questions in subsequent work on these stabilization indices:
\begin{enumerate}
\item Do formulas for $\Stab(I)$ exist for squarefree monomial ideals?  Do they relate to the dimensions of the facet complex or the Stanley-Reisner complex?
\item Does $\Stab(I_{\Delta})$ have a topological interpretation in terms of $\Delta_{\text{pol}(I^n)}$, the Stanley-Reisner complex of the polarization of $I^n$?
\item Does there exist a class of ideals for which the $D$ produced in Theorem~\ref{thm:bettistabilization} is the sharp bound, i.e. $D=\Stab(I)$?
\end{enumerate}
Aside from the stabilization index, the shapes of chain of Betti tables leading up to the stabilized Betti table appear fairly interesting.  Generally, the shapes of Betti tables of powers of homogeneous equigenerated ideals seem to be unimodal, in the following sense:
\begin{conjecture}  Let $I\subseteq R$ be an equigenerated homogeneous ideal generated in degree $r$.  Then for each pair of indices $(i,j)$ there exist $1\leq D_1\leq D_2\leq \infty$ such that for all $d$ with $D_1\leq d\leq D_2$,
$$\betti{i}{j+dr}(I^d)\neq 0$$
and for all $d<D_1$ or $D_2<d$,
$$\betti{i}{j+dr}(I^d)=0.$$
\end{conjecture}
Proving this conjecture would require a better understanding of the modules $M_{ij}$ described above.  These $M_{ij}$ seem to carry interesting structure, and investigating the connections between $M_{ij}$ and the geometry of the ideal $I$ and its Rees algebra $\rees(I)$ is another area of future interest.
\medskip

\textbf{Acknowledgements.}
The author would like to thank her advisor, Mike Stillman, for his many helpful conversations and comments throughout the course of this work.  This paper would not have been possible without his many suggestions.

\bibliographystyle{amsalpha}
\bibliography{bibliography}

\providecommand{\bysame}{\leavevmode\hbox to3em{\hrulefill}\thinspace}
\providecommand{\MR}{\relax\ifhmode\unskip\space\fi MR }
\providecommand{\MRhref}[2]{%
  \href{http://www.ams.org/mathscinet-getitem?mr=#1}{#2}
}
\providecommand{\href}[2]{#2}
\begin{thebibliography}{CHT99}

\bibitem[Bor09]{MR2604920}
Keivan Borna, \emph{On linear resolution of powers of an ideal}, Osaka J. Math.
  \textbf{46} (2009), no.~4, 1047--1058. \MR{2604920}

\bibitem[CHT99]{MR1711319}
S.~Dale Cutkosky, J{\"u}rgen Herzog, and Ng{\^o}~Vi{\^e}t Trung,
  \emph{Asymptotic behaviour of the {C}astelnuovo-{M}umford regularity},
  Compositio Math. \textbf{118} (1999), no.~3, 243--261. \MR{1711319
  (2000f:13037)}

\bibitem[Con06]{MR2184787}
Aldo Conca, \emph{Regularity jumps for powers of ideals}, Commutative algebra,
  Lect. Notes Pure Appl. Math., vol. 244, Chapman \& Hall/CRC, Boca Raton, FL,
  2006, pp.~21--32. \MR{2184787 (2007c:13017)}

\bibitem[R{\"o}m01]{MR1895463}
Tim R{\"o}mer, \emph{Homological properties of bigraded algebras}, Illinois J.
  Math. \textbf{45} (2001), no.~4, 1361--1376. \MR{1895463 (2003d:13015)}

\bibitem[Stu00]{Sturmfels:2000fk}
B~Sturmfels, \emph{Four counterexamples in combinatorial algebraic geometry},
  Journal of Algebra \textbf{230} (2000), 282--294.

\bibitem[Swa97]{MR1428875}
Irena Swanson, \emph{Powers of ideals. {P}rimary decompositions, {A}rtin-{R}ees
  lemma and regularity}, Math. Ann. \textbf{307} (1997), no.~2, 299--313.
  \MR{1428875 (97j:13005)}

\bibitem[Vas94]{MR1275840}
Wolmer~V. Vasconcelos, \emph{Arithmetic of blowup algebras}, London
  Mathematical Society Lecture Note Series, vol. 195, Cambridge University
  Press, Cambridge, 1994. \MR{1275840 (95g:13005)}

\end{thebibliography}
\end{document}